\documentclass[12pt]{amsart}
\usepackage{amsmath, amsthm, amssymb,cite}
\usepackage{fullpage}
\usepackage{color}
\usepackage{hyperref}
\usepackage{soul}

\newtheorem{theorem}{Theorem}
\newtheorem{lemma}{Lemma}

\newtheorem{corollary}[lemma]{Corollary}

\numberwithin{lemma}{section}

\numberwithin{equation}{section}

\newcommand{\R}{{\mathbb R}}

\newcommand{\Z}{{\mathbb Z}}

\renewcommand{\R}{\mathbf R}

\newcommand{\W}{{\mathbf W}}

\renewcommand\R{\mathbf R}

\begin{document}
\vspace*{-.2in}

\title{No solitary waves in 2-d  gravity and capillary waves in deep water }

\author{Mihaela Ifrim}
\address{Department of Mathematics, University of Wisconsin, Madison}
\email{ifrim@math.wisc.edu}

\author{ Daniel Tataru}
\address{Department of Mathematics, University of California at Berkeley}
\email{tataru@math.berkeley.edu}

\begin{abstract}
A fundamental question in the study of water waves is the existence and stability of 
solitary waves. Solitary waves have been proved to exist and have been studied in many 
interesting situations, and often arise from the balance of different forces/factors 
influencing the fluid dynamics, e.g. gravity, surface tension or the fluid bottom.
 However, the existence of solitary waves has remained an open problem
in two of the simplest cases, namely for either pure gravity waves 
or pure capillary waves in infinite depth. 

In this article we settle both of these questions in two space dimensions.
 Precisely, we consider  incompressible, irrotational, 
infinite depth water wave equation, either with 
gravity and without surface tension, or without gravity but with surface tension. 
In both of these cases we prove that there are no solitary waves. 
\end{abstract}

\maketitle

\section{Introduction}

The water wave equations describe the motion of the free surface of an
inviscid incompressible fluid.  Under the additional assumption of
irrotationality, Zakharov~\cite{zak} observed that this motion can be
described purely by the evolution of the fluid surface.

Solitary waves are waves which move with constant velocity while
maintaining a fixed asymptotically flat profile. The question of
existence of solitary waves is a classical problem, going back to
Russell's 1844 experimental observation of such waves.  Equally
interesting is the periodic setting, where one instead seeks periodic
traveling waves.

Water waves come in many flavors depending on a number of physical
parameters, most notably gravity, surface tension (capillarity) and
fluid depth.  Solitary waves have been discovered in many of these
settings in both two and higher dimensions, beginning with the results in finite depth in
\cite{MR0065317}, \cite{MR0445136} and \cite{MR629699}, followed by
the bifurcation result in \cite{MR1133302}. An extensive literature
exists now on this subject, including both results for gravity and for
gravity/capillary waves, see for instance
\cite{MR963906,MR2379653,MR2097656,MR1378603,MR1720395,MR2263898,
MR1949969,MR2867413}
and the review in \cite{MR2097656}. This famously includes for
instance the Stokes waves of greatest height, which have an angular
crest; these were conjectured in \cite{Stokes} and proved to exist in
\cite{MR513927}, \cite{MR666110}.  Within this family of problems, one of the
  most difficult cases turned out to be the case of deep water.

For water waves in deep water, solitary waves
have been recently proved to exist provided that both gravity and
surface tension are present, see \cite{MR1423002,MR2069635,MR2073504,MR2847283}, 
following numerical work in 
\cite{MR988299,MR990170}. However, in the seemingly
simpler cases where exactly one of these forces is active, the
problem has remained largely open, and only partial nonexistence results are known
\cite{MR1949966,MR1455326,MR2993054}.

The goal of the present work is to settle both of these problems in
the two dimensional setting. We show that with either gravity and no
surface tension, or with surface tension but no gravity, no solitary
waves exist in infinite depth.  It remains an open problem to prove a
similar result in higher dimension.

 The difficulty in both of these problems is that the water wave
 equations are not only fully nonlinear but also nonlocal. The
 equations we consider do have a lot of structure, as well as a
 scaling symmetry. However, taking advantage of this structure in the
 classical Eulerian formulation seems untractable at this time,
 particularly for large data solutions. Instead, in the present paper
 we use the holomorphic (conformal) formulation of the equations,
 which is specific to two dimensional problems. In this setting we
 eventually arrive at a simpler, symmetric formulation of the solitary
 wave equations, though still nonlinear and nonlocal.  Our final
 argument is vaguely reminiscent of proofs of the absence of embedded
 resonances for elliptic operators, with an added twist which is due
 to the nonlocality.

\subsection{ The Eulerian equations}
We consider the incompressible, infinite depth water wave
equation in two space dimensions, either with gravity but no surface tension,
or without gravity but with surface
tension. This is governed by the incompressible Euler's equations
with boundary conditions on the water surface. 

We first describe the equations. We denote the water domain at time $t$ by 
$\Omega(t)$, and the water
surface at time $t$ by $\Gamma(t)$. We think of $\Gamma(t)$ as being
either asymptotically flat at infinity or periodic. The fluid velocity
is denoted by $u$ and the pressure is $p$. Then $u$ solves the 
Euler's equations inside $\Omega(t)$, 
\begin{equation}
\left\{
\begin{aligned}
& u_t + u \cdot \nabla u = -\nabla p  -  g j
\\
& \text{div } u = 0
\\
& u(0,x) = u_0(x),
\end{aligned}
\right.
\end{equation}
while on the boundary we have the dynamic boundary condition
\begin{equation}
 p = -2 \sigma {\bf H}  \ \ \text{ on } \Gamma (t ),
\end{equation}
and the kinematic boundary condition 
\begin{equation}
\partial_t+ u \cdot \nabla \text{ is tangent to } \bigcup \Gamma (t).
\end{equation}
Here $g$ represents the gravity, ${\bf H}$ is the mean curvature of the 
boundary and $\sigma$ represents the surface tension.

Under the additional assumption that the flow is irrotational, we can
write $u$ in terms of a velocity potential $\phi$ as $u = \nabla \phi$, 
where $\phi$ is harmonic within the fluid domain, with appropriate decay at infinity.
Thus $\phi$ is determined by its trace on the free boundary
$\Gamma (t)$. Denote by $\eta$ the height of the water surface as a
function of the horizontal coordinate. Following Zakharov~\cite{zak}, we introduce 
$\psi =\psi (t,x)\in \mathbb{R}$ to be the trace of the velocity potential $\phi$ on the boundary,
$\psi (t,x)=\phi (t,x,\eta(t,x))$. Then the fluid dynamics can be expressed in terms of a 
one-dimensional evolution of  the pairs of variables
$(\eta,\psi)$, namely 
\begin{equation}
\left\{
\begin{aligned}
& \partial_t \eta - G(\eta) \psi = 0 \\
& \partial_t \psi + g \eta - \sigma {\bf H}(\eta)  +\frac12 |\nabla \psi|^2 - \frac12 
\frac{(\nabla \eta \cdot \nabla \psi +G(\eta) \psi)^2}{1 +|\nabla \eta|^2} = 0 .
\end{aligned}
\right.
\end{equation}
Here $G$ represents the Dirichlet to Neumann map associated to the fluid domain.
This is Zakharov's  Eulerian formulation of the gravity/capillary water wave equations. One limitation of this 
formulation is that it assumes that the fluid surface is a graph.

It is known since Zakharov~\cite{zak} that the water-wave 
system is Hamiltonian, where the Hamiltonian (conserved energy) is given by
\[
\mathcal{H}(\eta ,\psi)= \frac{g}{2} \int_{\mathbb{R}}\eta^2\, dx +\sigma \int_{\mathbb{R}}
\left( \sqrt{1+\eta_x^2}-1 \right)dx +\frac{1}{2}\int_{\mathbb{R}}\int_{-\infty}^{\eta} \vert \nabla_{x,y}\phi\vert^2 \, dx\, dy.
\]

The momentum is another conserved quantity which arises by Noether's
theorem as a consequence of the fact that the problem is invariant
with respect to horizontal translations (see Benjamin and Olver
\cite{MR688749} for a complete study of the invariants and symmetries
of the water-wave equations):
\[
\mathcal{M}=\int_\mathbb{R}\int_{-\infty}^{\eta(t,x)}\phi_x (t,x,y) \, dy\,  dx.
\]

Solitary waves are solutions to the water wave equations which have a constant profile which moves 
with constant speed,
\[
(\eta,\psi)(x,t) = (\eta,\psi)(x-ct),
\]
and which decay at infinity. Here $c$ represents the horizontal speed of propagation for the free surface.
The horizontal fluid speed, however, is in general not equal to $c$.

\subsection{ Scale invariance and critical regularity}

We now discuss natural regularity assumptions for solitary waves in the Eulerian setting.
The energy space for solutions $(\eta,\psi)$, provided by the Hamiltonian, is different for gravity 
versus capillary waves,
\[
E^g := L^2 \times \dot H^\frac12, \qquad  E^c := \dot H^1 \times \dot H^\frac12.
\]

We will not a-priori assume that a solitary wave has to have finite
energy, though a-posteriori this will be established as an
intermediate step within the proof.  However, in order to study the
existence of solitary waves, it is natural to consider solutions which
have at least critical regularity.  This is best understood from the
scaling properties of the equation, which also differ depending on
whether we consider gravity or capillary waves.

For gravity waves, the equations admit the scaling law 
\[
(\eta(x,t), \psi(x,t)) \to (\lambda^{-1} \eta(\lambda x, \lambda^{\frac12} t), \lambda^{-\frac32}\lambda x, \lambda^{\frac12} t).
\]
Thus one might be led to consider $(\eta,\psi)$ in the space $\dot
H^\frac32 \times \dot H^2$.  However, this is not accurate due to the
fact that the water wave system, viewed  at the linearized level, is
degenerate hyperbolic and in non diagonal form. Because of this, one
needs to work instead with \emph{Alihnac's good variable}, or in other words the 
diagonal variable,  which in
this context is best understood in terms of the differentiated
variables, namely $(\eta_x, \nabla \phi_{|y = \eta})$.  A scale
invariant space for these variables is for instance
\[
E^{g}_{crit} := \dot H^\frac12   \times \dot H^1 .
\]
Here we could also add the $L^\infty$ norm for $\eta_x$, which would provide 
the Lipschitz property for the free surface. 

For capillary waves, the equations admit the scaling law 
\[
(\eta(x,t), \psi(x,t)) \to (\lambda^{-1} \eta(\lambda x, \lambda^{\frac32} t), \lambda^{-\frac12}\lambda x, \lambda^{\frac32} t).
\]
Here the spaces are the same whether we look at $(\eta_x, \nabla \phi_{|y = \eta})$ or $(\eta_x,\psi_x )$, namely
\[
E^{g}_{crit} := \dot H^\frac12 \times L^2.
\]

Our a-priori assumptions on the solitary waves will be slightly
stronger than the above critical norms, but still at critical scaling;
this is in part due to the need to define the conformal map in a way
that does not looses the critical regularity. In particular the $\dot
H^\frac12$ norm for $\eta_x$ is not good enough, as it does not yield
a Lipschitz bound on $\Gamma$. Instead, we shall use a stronger bound,
namely the norm in the homogeneous Besov space $\dot
B^{\frac12}_{2,1}$, where the dyadic pieces of $\eta_x$ are still
measured in $\dot H^{\frac12}$ but summed in $l^{1}$.  Precisely, this
is defined in terms of a standard dyadic Littlewood-Paley
decomposition
\[
1 = \sum_{k \in \Z} P_k,
\]
where the projectors $P_k$ select functions with frequencies $\approx 2^k$, as
\[
\| u \|_{\dot B^{\frac12}_{2,1}} = \sum_{k \in \Z} 2^{\frac{k}2} \| P_k u\|_{L^2}. 
\]

\subsection{ Solitary waves and the non-existence result}

Now we are in a position to provide the Eulerian formulation of our main results. 
We begin with the gravity waves:

\begin{theorem}
\label{t:g-euler}
The two dimensional gravity wave equation in deep water admits no solitary waves 
$(\eta,\psi)$ with  critical regularity $\eta_x \in \dot B^{\frac12}_{2,1},  \nabla \phi_{|y = \eta}\in 
 \dot H^1$.
\end{theorem}

Here it is not essential to specify also the regularity of the
velocity $\nabla \phi$ on the free surface, as for solitary waves this
is immediately seen to be a consequence the regularity of $\eta$.

In earlier work it was shown\footnote{This is both for the two and the
  three dimensional problem.} in \cite{MR1949966} that there are no
solitary waves with either positive elevation $(\eta \geq 0)$ or
negative elevation $(\eta \leq 0)$ in both two and three dimensions, 
and more recently, that there are no two dimensional 
solitary waves with at least $|x|^{-1-\epsilon}$ decay at infinity,
see \cite{MR1455326,MR2993054}.  The expansion at infinity is
discussed in greater detail in \cite{2016arXiv160401092W}.

Compared with these earlier works, here we impose no sign condition on
$\eta$, as well as no decay condition on $\eta$; indeed, our only
``decay'' assumption is $\eta_x \in  \dot B^{\frac12}_{2,1}$. In view of the embedding 
 $\dot B^{\frac12}_{2,1} \subset L^\infty$ this does require $\eta$ to be Lipschitz continuous
and thus of at most linear growth. However,  it allows for functions 
with almost linear growth at infinity, which is the natural critical assumption,
and far from  any actual decay condition.

Here it is important to work at low regularity, as the solitary wave
equations are not necessarily elliptic, and in particular may allow
angular crests, e.g. as in the Stokes waves which have $2\pi/3$
angles. It is known that such crests can locally occur only at points
of maximal elevation, and that the crest points are stationary point for the flow,
where the fluid speed and the solitary wave speed coincide.
Away from these stationary points for the flow, the 
 equations for the solitary waves are seen to be elliptic, and  the 
solitary waves must be $C^\infty$. 

In this work we do not impose any condition to guarantee the absence
of stationary points, though the above Besov regularity property does
preclude angular crests. However, while proving the result we give a
stronger form of it in conformal coordinates, see
Theorem~\ref{t:g-hol}, which also allows for angular crests. Then, at
the end of Section~\ref{s:g}, we show that this form excludes crests
with angles $\theta \in (\pi/2,\pi)$, and in particular it excludes
Stokes types crests, which are the only ones consistent with the
solitary wave equations.

\bigskip

Next we continue with our result for capillary waves, where we use similar a-priori bounds.
Here there are no  prior results in this direction, as far as we are aware.

\begin{theorem}
\label{t:c-euler}
The two dimensional capillary wave equations in deep water admits no
solitary waves $(\eta,\psi)$ with critical regularity $\eta_x \in
\dot B^{\frac12}_{2,1}$, $\nabla \phi_{|y = \eta}\in 
 L^2$.
\end{theorem}

Our approach for both of these problems relies on the use of the
holomorphic formulation of the equations, which uses conformal
coordinates within the fluid domain. This approach\footnote{Babenko
  does not phrase his approach explicitly in terms of conformal
  coordinates, but the outcome is nevertheless the same.}  has been
initially used precisely in the study of periodic traveling waves in
work of Babenko~\cite{MR898306,MR899856}, and later for solitary waves
in \cite{MR1133302,MR1764946,MR2993054} and many other works.  It has also 
been implemented in the study of the dynamical problem for gravity waves in work of
Wu~\cite{wu2}, Dyachenko-Kuznetsov-Spector-Zakharov~\cite{zakharov},
and more recently by the authors and Hunter~\cite{HIT}, as well as for
capillary waves by the authors in \cite{ITCapillary}.  The equivalence 
between the solitary wave equation and the Babenko equation has also been explored in 
\cite{MR2458311,MR2993054}.

 In addition to the Babenko equations \cite{MR898306}, an alternate method is to use a
different conformal transformation, called a hodograph transform, see
e.g. \cite{MR1423002} and references therein. This last method applies 
only for the study of  steady flows, and not for the dynamical problem.

In the present work we derive a set of equations for steady gravity  waves in
holomorphic coordinates, which are easily seen to imply the Babenko
equations for gravity waves; however, the converse is far less
obvious. The equivalence of the Eulerian and the conformal formulation
of the solitary wave equations has been extensively studied, see \cite{MR2458311} and the 
references therein.

Similarly, we also produce a related set of equations for
capillary waves.  This is done in the next section, where we also
recall the holomorphic form of the water wave evolution.

In the last two sections we prove the solitary wave nonexistence
result, including also a statement of the main results which applies
directly in holomorphic coordinates. This is a stronger statement, as
in particular it does not require the free fluid surface to be a
graph, and also precludes angular crests in the case of gravity waves.

\subsection*{Acknowledgments} The first author
was partially supported by a Clare Boothe Luce Professorship. The second  author was
partially supported by the NSF grant DMS-1800294 as well as by a Simons Investigator
grant from the Simons Foundation.

\section{Water waves in holomorphic coordinates and the 
Babenko equations }

\subsection{Holomorphic coordinates}
One main difficulty in the study of the above equations is the
presence of the Dirichlet to Neumann operator $\mathcal D$, which
depends on the free boundary.  This is one of the motivations for our
choice to use a different framework to study these equations, namely
the holomorphic coordinates. In the context of the dynamical problem
these were first introduced by Ovsiannikov~\cite{ov}, and further
developed by Wu~\cite{wu} and
Dyachenko-Kuznetsov-Spector-Zakharov~\cite{zakharov}, and were heavily
exploited by the authors in earlier work \cite{HIT}, \cite{IT-global}
in the context of two dimensional gravity water waves and
\cite{ITCapillary} for two dimensional capillary waves.

The holomorphic coordinates are defined via a conformal map
$Z: \mathbb{H} \to \Omega(t)$, where $\mathbb{H}$ is the
lower half plane, $\mathbb{H}:=\left\lbrace \alpha -i\beta\ : \ \beta
  > 0 \right\rbrace $. Here $Z$ maps the real line into the
free boundary $\Gamma (t)$. Thus the real coordinate $\alpha$
parametrizes the free boundary, which is denoted by $Z(t,\alpha)$. 
Requiring $Z$ to satisfy the condition
\[
\lim_{\alpha \to \pm \infty} Z_\alpha = 1
\]
uniquely determines it modulo horizontal translations.
If in addition the fluid surface is assumed to be asymptotically flat then
one can remove this remaining degree of freedom 
in the choice of  the parametrization by imposing a stronger boundary
condition at infinity,
\[
\lim_{\alpha \to \pm \infty} Z(t,\alpha) - \alpha = 0.
\]

In order to describe the velocity potential $\phi$ we use the
 function $Q= \phi+iq$ where $q$ represents the harmonic
 conjugate of $\phi$ (the stream function). Both functions $Z$ and $Q$ admit Lipschitz
 holomorphic extensions into the lower half plane, which implies that
 their Fourier transforms are supported in $(-\infty,0]$.  By a slight abuse of terminology we call
 such functions holomorphic functions. They can be described by the relation $Pf = f$,
where  $P$ represents the projector operator  to negative frequencies. For later use we recall
that $P$ can be represented as
\[
P:=\frac{1}{2}\left(I-iH\right),
\] 
where $H$ is the classical Hilbert transform on the real line.

The water wave equations will  define a flow in the class of holomorphic functions
for the pair of variables $(W,Q)$ where $W := Z-\alpha$.

For a full derivation of the holomorphic form of gravity/capillary water waves we refer 
the reader  to the paper \cite{HIT} for gravity waves, and \cite{ITCapillary} for capillary waves.
The equations have the form
\begin{equation}
\label{ww2d}
\left\{
\begin{aligned}
& W_t + F (1+W_\alpha) = 0\\
& Q_t + F Q_\alpha  - i g W + P\left[ \frac{|Q_\alpha|^2}{J}\right] -\sigma P\left[ \frac{-i}{2+W_{\alpha}+\bar{W}_{\alpha}}\frac{d}{d\alpha}\left( \frac{W_{\alpha}-\bar{W}_{\alpha}}{\vert 1+W_{\alpha}\vert }\right)\right]  = 0, \\
\end{aligned} 
\right.
\end{equation}
where
\begin{equation}
\label{rww2d}
 F := P\left[\frac{Q_\alpha - \bar Q_\alpha}{J}\right], \qquad J := |1+W_\alpha|^2.
 \end{equation}
Simplifying we obtain the fully nonlinear system
\begin{equation}\label{ngst}
\left\{
\begin{aligned}
& W_t + F (1+W_\alpha) = 0 \\
& Q_t + F Q_\alpha -i g W + P\left[ \frac{|Q_\alpha|^2}{J}\right] +i\sigma P\left[ \frac{W_{\alpha \alpha}}{J^{1/2}(1+W_{\alpha})}-\frac{\bar{W}_{\alpha \alpha}}{J^{1/2}(1+\bar{W}_{\alpha})}\right]  = 0.\\
\end{aligned} 
\right.
\end{equation}

This is a still Hamiltonian system, where the Hamiltonian  has a simple form 
when expressed in holomorphic coordinates:
\begin{equation*}
\mathcal{H}(W, Q)=\int \Im (Q\bar{Q}_{\alpha})
+2\sigma \left( J^{\frac{1}{2}} - 1 -\Re W_\alpha\right)  + g \left(  |W|^2 - 
\frac12 (\bar W^2 W_\alpha + W^2 \bar W_\alpha)\right)  d\alpha.
\end{equation*}

The invariance with respect to horizontal translations leads by Noether's theorem
to the conservation of momentum,
\begin{equation*}
\mathcal{M}(W,Q)=i\int  \bar{Q}W_{\alpha}-Q\bar{W}_{\alpha} \, d\alpha .
\end{equation*}

One downside of working in holomorphic coordinates is that  
here the symplectic form is not as simple as in the Eulerian case.
Fortunately this is not needed in the present paper.

Since $W$ and $Q$ only appear in differentiated form in the equations above,
differentiating with respect to $\alpha$ yields a self contained 
system for $(W_\alpha,Q_\alpha)$. This system is best expressed 
using \emph{ Alihnac's good variable}, which in this case has the form
\[
 (\W,R) = \left( W_\alpha,  \frac{Q_\alpha}{1+W_\alpha} \right).
\]
Here $1+\W$ describes the slope of the free surface and $R$ is simply the velocity vector
in complex notation, both in the holomorphic parametrization. 
The system for $(\W,R)$ has the form
\begin{equation} \label{diff-ngst}
\left\{
\begin{aligned}
 & \W_{ t} + b \W_{ \alpha} + \frac{(1+\W) R_\alpha}{1+\bar \W}   =  (1+\W)M
\\
& R_t + bR_\alpha +\frac{ i(g+ a)}{1+\W} + i\sigma \frac{1}{1+\W}P\left[ \frac{\W_{ \alpha}}{J^{1/2}(1+\W)}\right]_{\alpha}  = i \sigma P \left[\frac{\bar{\W}_{ \alpha}}{J^{1/2}(1+\bar{\W})}\right]_{\alpha},
\end{aligned}
\right.
\end{equation}
where the real advection velocity $b$ is given by 
\begin{equation}\label{b-def}
b := 2 \Re P\left[ \frac{Q_\alpha}{J}\right].
\end{equation}
The real {\em frequency-shift} $a$ is given by
\begin{equation}
a := i\left(\bar P \left[\bar{R} R_\alpha\right]- P\left[R\bar{R}_\alpha\right]\right),
\label{defa}
\end{equation} 
and  the auxiliary function $M$ is given by
\begin{equation}\label{M-def}
M :=  \frac{R_\alpha}{1+\bar \W}  + \frac{\bar R_\alpha}{1+ \W} -  b_\alpha =
\bar P [\bar R Y_\alpha- R_\alpha \bar Y]  + P[R \bar Y_\alpha - \bar R_\alpha Y].
\end{equation}
Here we used the notation $Y:=\dfrac{\W}{1+\W}$. Primarily, the equations \eqref{ngst} 
will be used in the sequel. However the above discussion motivates the fact that all the results here 
can also be phrased  in terms of the differentiated variables.

\subsection{ The regularity of $W_{\alpha}$ and $R$}

Here we transfer the a-priori regularity of the Eulerian variables to the holomorphic 
variables $(W,Q)$. 

\begin{lemma}
\noindent 
a)  Assume that $\eta_x \in \dot B^{\frac12}_{2,1}$. Then
$W_{\alpha} \in \dot B^{\frac12}_{2,1}$ and the following inequality holds
\begin{equation}\label{bi-lip} 
   1+\Re W_{\alpha} \geq \delta > 0.
\end{equation}
$b)$ Assume in addition that the trace $\nabla \phi |_{y=\eta}$ of the
velocity on the free surface is in $L^2$, respectively $\dot
H^1$. Then $R$ also has regularity $L^2$, respectively $\dot
H^1$.
\end{lemma}

\begin{proof}

To obtain the conformal map $Z$ and thus $W_{\alpha}$
 it suffices to construct the harmonic function $\beta$ inside 
the fluid domain $\Omega(t)$. 
Then $\alpha$ is the harmonic conjugate of $\beta$, normalized so that 
it vanishes on the top. 
Here $\beta$ is a negative function, which solves the Dirichlet problem
\begin{equation}
\label{laplace}
\left\{
\begin{aligned}
& -\Delta \beta = 0 \qquad \text{ in } \Omega(t)
\\
& \ \ \ \beta = 0 \qquad \quad \, \text{  in } \Gamma(t)
\\
& \lim_{(x,y) \to \infty } \frac{\beta - y}{|x|+|y|} = 0 .
\end{aligned}
\right.
\end{equation}

\bigskip
{\bf STEP 1.} We first prove the uniform bound \eqref{bi-lip}. 
A direct computation using the inverse function theorem shows that
\[
1+ \Re W_{\alpha} =  \frac{ \partial_y \beta}{|\nabla_{(x,y)} \beta|^2},
\]
where due to the homogeneous boundary condition we have
\[
|\beta_x| \lesssim |\beta_y|.
\]
Thus we need to show that $\vert \beta_y \vert \lesssim 1$ on the top,
or equivalently that $|\nabla \beta| \lesssim 1$, or equivalently that
 \begin{equation} \label{beta-size}
\vert \beta(x,y) \vert \approx |y - \eta|.
\end{equation}

Our starting point for this is the authors' result in \cite{AIT},
Proposition 3.3, which asserts that if $\eta_x$ is small in $\dot
B^{\frac12}_{2,1}$ then $\W$ is small in the same space. This
further implies that $\W$ has a small uniform bound, and thus
$|1+\W\approx 1$. For $\beta$ this shows that
\eqref{beta-size} holds.

For our proof here we will combine this small data result with the
maximum principle in two steps. The maximum principle shows that if
$\eta_1, \eta_2$ are two elevation functions so that $\eta_1 \leq
\eta_2$ then for the corresponding holomorphic coordinates $\beta_1$,
$\beta_2$ we have $\vert \beta_1 \vert \leq \vert \beta_2 \vert$.

We first regularize $\eta$ at a small frequency scale
$\lambda \ll 1$, namely $\eta_{\leq \lambda}$. If $\lambda$ is small
enough then $\eta_{\leq \lambda}$ is small in $\dot B^{\frac12}_{2,1}$, so
the result in \cite{AIT} applies. Thus, the corresponding function
$\beta_\lambda$ satisfies
\[
\beta_\lambda (x,y) \approx |y - \eta_{\leq \lambda}|.
\]
Here $\beta_{\lambda}$ is the harmonic function solving \eqref{laplace}, where the defining function for $\Gamma$ is $\eta_{\leq \lambda}$. 

On the other hand, by Bernstein's 
inequality, we have 
\[
| \eta - \eta_{\leq \lambda}| \lesssim 1.
\]
Comparing $\eta < \eta_{\leq \lambda} + C$, by the maximum principle it follows that 
\[
\vert \beta(x,y) \leq \vert \beta_\lambda(x,y-C) + C \vert ,
\]
and thus
\[
\vert \beta(x,y)\vert  \leq |y - \eta(x)| + 2C .
\]
This provides the desired bound \eqref{beta-size} at distance $\gtrsim 1$ 
from $\Gamma$, 
and localizes the problem to the 
unit spatial scale.

For the next step of the proof of \eqref{beta-size}, we fix $x_0 \in \R$ and seek to
estimate $\beta$ near $x_0$. After a rotation (which preserves the
$B^{\frac12}_{2,1}$ regularity), we can assume without any restriction
in generality that $\eta'(x_0) = 0$. Then, by localizing $\eta$ near $x_0$, we
can find a function $\eta_1$ which is small in $B^{\frac12}_{2,1}$ so that 
\[
\eta_1(x) = \eta(x_0) + (x-x_0)^2, \qquad |x-x_0| \ll 1.
\]
Then comparing the corresponding  functions $\beta$ and $\beta_1$ in a 
set 
\[
\{ |x-x_0| \leq \epsilon, \quad  \eta(x) -\epsilon \leq y \leq \eta(x) \} ,
\]
by the maximum principle it follows that in this set we have
\[
\vert \beta \vert \leq C \vert \beta_1 \vert .
\]
In particular at $x = x_0$ we have 
\[
\vert \beta(x_0,y) \vert  \leq  \vert \beta_1(x_0,y) \vert \lesssim \vert \eta_1(x_0)-y \vert \lesssim \vert\eta(x_0)-y\vert ,
\]
as needed.

\bigskip

{\bf STEP 2.} Here we transfer the Sobolev regularity from $\eta_x$ to
$W_{\alpha}$, and from $\nabla \phi |_{y=\eta}$ to $R$.  
On the top we have the relations 
\[
\Im Z(\alpha) = \eta(x), \qquad R(\alpha) = \phi_x + i \phi_y,
\]
where the relation between the two parametrizations is given by 
\[
x = \Re Z (\alpha).
\]
By \eqref{bi-lip} this is a bi-Lipschitz map, so we have the 
straightforward Sobolev norm equivalence
\[
\| R \|_{L^2_\alpha} \approx \| \nabla \phi |_{y=\eta}\|_{L^2_x}, \qquad \| R \|_{\dot H^1_\alpha} 
\approx \| \nabla \phi |_{y=\eta}\|_{\dot H^1_x.}
\]
Interpolating between $L^2$ and $\dot H^1$, we also obtain 
the Besov norm equivalence
\[
\|\eta_x \|_{\dot B^{\frac12}_{2,1,x}} \approx \|\eta_x \|_{\dot B^{\frac12}_{2,1,\alpha}} 
\]
We still need to transfer the last bound to $W_{\alpha}$. For that we compute
\[
\eta_x = \frac{\Im Z_\alpha}{\Re Z_\alpha} .
\]
This implies that
\begin{equation}
\frac{\eta_x}{1+\eta_x^2} = - \Im \left( \frac{1}{Z_\alpha} \right) = 
\Im \left( \frac{W_{\alpha}}{1+W_{\alpha}}\right) 
\end{equation}
Now we use the algebra property of the Besov space $\dot B^{\frac12}_{2,1}$,
as well as the fact
 that standard Moser estimates hold for this space. For later reference we state 
the result in the following 
\begin{lemma}
a) The space $\dot B^{\frac12}_{2,1}$ is an algebra.

b) Let $W \in \dot B^{\frac12}_{2,1}$, and $G$ a smooth function with $G(0)=0$.
Then we have the Moser estimate
\begin{equation}
\| G(W)\|_{ \dot B^{\frac12}_{2,1}} \lesssim C(\|W\|_{L^\infty}) \| W\|_{ \dot B^{\frac12}_{2,1}}.
\end{equation}
\end{lemma}
This is a standard result and the proof is omitted. Such a property is
proved for instance in $\dot H^\frac12 \cap L^\infty$ in \cite{HIT};
the Besov case is completely similar.

This property  implies that
\[
\left\| \frac{\eta_x}{1+\eta_x^2} \right\|_{\dot B^{\frac12}_{2,1,\alpha}} \lesssim \| \eta_x \|_{\dot B^{\frac12}_{2,1,\alpha}}
C(\| \eta_x \|_{L^\infty})
\]
Thus we obtain 
\[
\Im \left( \frac{W_{\alpha}}{1+W_{\alpha}}\right) \in  \dot B^{\frac12}_{2,1,\alpha}
\]
But the function $\dfrac{W_{\alpha}}{1+W_{\alpha}}$ is bounded holomorphic in the lower half-plane 
with decay at infinity, so we have the usual relation between its real and imaginary part,
\[
\Re \frac{W_{\alpha}}{1+W_{\alpha}} = H \Im \frac{W_{\alpha}}{1+W_{\alpha}}
\]
Thus we obtain
\[
 \frac{W_{\alpha}}{1+W_{\alpha}} \in  \dot B^{\frac12}_{2,1,\alpha}
\]
Here we know that $1+ W_{\alpha}$ is bounded away from zero and thus $W_{\alpha}$ is bounded. 
This, again by the Moser property, yields 
\[
W_{\alpha} \in  \dot B^{\frac12}_{2,1,\alpha}
\]
as needed.

\end{proof}

\subsection{ Solitary waves and the Babenko equations}

Solitons are solutions for \eqref{ww2d} of the form
$(Q(\alpha-ct),W(\alpha-ct))$. Here the sign of $c$ is not important
due to the time reversal symmetry $(W(t,\alpha), Q(t,\alpha)) \to
(W(-t,\alpha),-Q(-t,\alpha))$. Substituting in the equations
\eqref{ngst} we obtain the system
\begin{equation}
\label{soliton}
\left\{
\begin{aligned}
& -cW_{\alpha} + F (1+W_\alpha) = 0\\
&- cQ_\alpha + F Q_\alpha  - ig W + P\left[ \frac{|Q_\alpha|^2}{J}\right] -\sigma P\left[ \frac{-i}{2+W_{\alpha}+\bar{W}_{\alpha}}\frac{d}{d\alpha}\left( \frac{W_{\alpha}-\bar{W}_{\alpha}}{\vert 1+W_{\alpha}\vert }\right)\right]  = 0. \\
\end{aligned} 
\right.
\end{equation}

\bigskip

Before pursuing this venue, we make a brief parenthesis to review the
Hamiltonian formalism and the corresponding more classical derivation
of the Babenko equations. While we will not rely on this in the present paper,
this is nevertheless an instructive exercise.   We recall that the
Hamiltonian and the horizontal momentum are given by
\begin{equation*}
\mathcal{H}(W, Q)=\int \Im (Q\bar{Q}_{\alpha})
+2\sigma \left( J^{\frac{1}{2}} - 1 -\Re W_\alpha\right)  + g \left(  |W|^2 - 
\frac12 (\bar W^2 W_\alpha + W^2 \bar W_\alpha)\right)  d\alpha,
\end{equation*}
respectively 
\begin{equation*}
\mathcal{M}(W,Q)=i\int  \bar{Q}W_{\alpha}-Q\bar{W}_{\alpha} \, d\alpha .
\end{equation*}
Since $\mathcal{M}$ is the generator of the group of horizontal translations with respect to the 
same symplectic form, it is natural to expect that solitary waves must formally solve the system
\begin{equation}\label{nother}
D\mathcal H(W,Q) = c D \mathcal M(W,Q).
\end{equation}
This can be interpreted as saying that solitary waves are critical points for the Hamiltonian on level sets 
of the momentum.  In this interpretation the velocity $c$ plays the role of the Lagrange multiplier.

We now use  the above relation to formally derive the Babenko's equations.
We first compute some simple variational derivatives.
\[
\frac{d \mathcal H}{d \bar Q} =  i Q_\alpha,  \qquad \frac{d\mathcal{M}}{d \bar Q} =i W_\alpha .
\]
Then from the first component of \eqref{nother} we obtain 
\begin{equation}\label{n1}
Q_\alpha = c W_\alpha.
\end{equation}

Next we have 
\[
\frac{d \mathcal M}{d \bar W} = iQ_\alpha.
\]
Thus we are left with a single equation for $W$, namely 
\[
\frac{d \mathcal H}{d \bar W} = ic^2 W_\alpha.
\]

We have 
\[
\frac{d \mathcal H}{d \bar W}  = -\sigma \partial_\alpha \left(\frac{1+W_\alpha}{|1+W_\alpha|}\right)  + g ( 
W - \bar W W_\alpha + W W_\alpha).
\]
This leads us to the following formulation of the Babenko equations:
\begin{equation}\label{n1a}
P\left[-\sigma \partial_\alpha \left(\frac{1+W_\alpha}{|1+W_\alpha|}\right)  + g ( 
W - \bar W W_\alpha + W W_\alpha)\right] = ic^2 W_\alpha.
\end{equation}
We remark that this is not exactly the standard formulation, which is done at the level 
of the function $\Im W$ which represents the elevation in the conformal parametrization.
In this article we prefer instead to work with the holomorphic variable $W_\alpha$.
Critically this allows us to  take advantage of the algebra structure for this class of functions.

We specialize the above equations  to gravity, respectively capillary waves.

\bigskip

(i) For gravity waves, the Babenko equations have the form
\begin{equation}\label{babenko-g}
gW - gP[\bar W W_\alpha - W W_\alpha] = i c^2 W_\alpha
\end{equation}
These admit a variational interpretation, namely 
as critical points for the reduced Hamiltonian (i.e. potential energy)
 \begin{equation*}
\mathcal{H}_0(W)=\int  g \left(  |W|^2 - \frac12 (\bar W^2 W_\alpha + W^2 \bar W_\alpha)\right)  d\alpha,
\end{equation*}
on level sets of the reduced  momentum
\begin{equation*}
\mathcal{M}_0(W)=i\int  \bar{W}W_{\alpha}-W\bar{W}_{\alpha} \, d\alpha .
\end{equation*}
\medskip

(ii) For capillary waves, the Babenko equations are
\begin{equation}\label{babenko-c}
-  P \partial_\alpha \left(\frac{1+W_\alpha}{|1+W_\alpha|}\right)= i c^2 W_\alpha
\end{equation}
As above, these can be seen as the equations for  critical points for the reduced Hamiltonian (i.e. potential energy)
\begin{equation*}
\mathcal{H}_0(W)=\int 2\sigma \left( J^{\frac{1}{2}} - 1 -\Re W_\alpha\right)    d\alpha,
\end{equation*}
on level sets of the reduced  momentum $\mathcal{M}_0$.

\bigskip
Rigorously deriving these equations in the low regularity setting via
the Hamiltonian formalism requires inverting the symplectic map,
which is not so simple in the holomorphic coordinate system. Also, we are not assuming enough 
low frequency regularity in order to guarantee that either the Hamiltonian or the momentum are finite.
Because of this we will derive these equations directly from
\eqref{soliton}. We begin with the first equation in \eqref{soliton},
which we rewrite as
\[
F = P\left[ \frac{Q_{\alpha} -\bar{Q}_{\alpha}}{J} \right]=\frac{cW_{\alpha}}{1+W_{\alpha}}.
\]
The expression inside the projection is imaginary, thus by taking the imaginary part of both sides, leads to the following equality
\[
\frac{1}{2}\frac{Q_{\alpha}-\bar{Q}_{\alpha}}{J}=\frac{1}{2}\frac{c\left(W_{\alpha}-\bar{W}_{\alpha} \right)}{J}.
\]
Hence,  $\Im Q_\alpha = c\Im W_\alpha$; therefore
\begin{equation}
\label{1}
Q_{\alpha} =cW_{\alpha}.
\end{equation}

To obtain the counterpart of the second Babenko equation \eqref{n1} we start from the
second equation in \eqref{soliton}, we substitute the expression of
$F$ from the first equation, use \eqref{1}, and equate the real part
of both sides of the resulting equation. This leads to the following relation
\[
- \frac{c^2}{2}(W_{\alpha}+\bar{W}_{\alpha} )+ \frac{c^2}{2}  \left( \frac{W_{\alpha}^2}{1+W_{\alpha}} + \frac{\bar{W}_{\alpha}^2}{1+\bar{W}_{\alpha}}  \right) +\frac{c^2}{2} 
 \frac{\vert W_{\alpha}\vert ^2}{J} -ig\frac{W-\bar{W}}{2}+ \frac{i\sigma}{1+W_{\alpha}}\partial_{\alpha}\left(\frac{1+W_{\alpha}}{\vert 1+W_{\alpha}\vert }\right)  = 0.
\] 
Further computations give
\[
-ig\frac{W-\bar{W}}{2}+\frac{i\sigma}{1+W_{\alpha}} \partial_{\alpha}\left(\frac{1+W_{\alpha}}{\vert 1+W_{\alpha}\vert }\right)  =  \frac{c^2}{2}\frac{ W_{\alpha} +\bar{W}_{\alpha} +W_{\alpha} \bar{W}_{\alpha}}{J} .
\]
We now specialize to gravity, respectively capillary waves.

\bigskip

(i) For gravity waves we have
\begin{equation}
\label{1g}
-ig(W-\bar{W})  =  c^2\frac{ W_{\alpha} +\bar{W}_{\alpha} +W_{\alpha} \bar{W}_{\alpha}}{\vert 1+W_{\alpha}\vert^2} .
\end{equation}
We multiply \eqref{1g} by $(1+W_{\alpha})$ to obtain
\begin{equation}
\label{1gg}
-ig(W-\bar{W})(1+W_{\alpha})  =  c^2\left[ W_{\alpha} +\frac{\bar{W}_{\alpha}}{1+\bar{W}_{\alpha}}\right].
\end{equation}
The right hand side  of \eqref{1gg} can be written as a holomorphic plus an antiholomorphic part, but the left hand side cannot be decomposed in a similar fashion. We therefore need to use the projection $P$ and multiply by $i$ to arrive at \eqref{babenko-g}.

\bigskip

(ii) For capillary water waves we have a self-contained equation for $\W=W_{\alpha}$,
\begin{equation}
\label{1c}
i\frac{\sigma}{1+\W}  \partial_{\alpha}\left(\frac{1+\W}{\vert 1+\W\vert }\right)  =  \frac{c^2}{2}\frac{ \W +\bar{\W} +\W \bar{\W}}{J}.
\end{equation}
As before, we multiply \eqref{1c} by  $(1+\W)$ to obtain
\begin{equation}
\label{1cc}
i\sigma  \partial_{\alpha}\left(\frac{1+\W}{\vert 1+\W\vert }\right)  =  c^2\left[ \W +\frac{\bar{\W}}{1+\bar{\W}}\right],
\end{equation}
and take the projection of the resulting equation to arrive at \eqref{babenko-c}.

In both cases, we remark that the above equations provide additional\footnote{Of course, the two different sets of equations
are ultimately equivalent, but proving this requires extra work.} information compared to the Babenko type equations
\eqref{babenko-g} and \eqref{babenko-c}. We will take advantage of this in the sequel.

\section{ Gravity waves}
\label{s:g}
Here we prove the desired nonexistence statement for solitary waves in
the case of the gravity waves, expressed in holomorphic
coordinates. Our starting point is the equation~\eqref{1g} for
$W$, which we recall here:
\begin{equation}\label{1g-re}
2 g \Im W = c^2 \frac{W_{\alpha} +\bar W_{\alpha} +|W_{\alpha}|^2}{|1+W_{\alpha}|^2}.
\end{equation}
For this equation we consider solutions $W$ whose derivative $W_{\alpha}$  has regularity as follows:
\begin{equation} \label{W-reg}
\begin{split}
W_{\alpha} \in \dot B^{\frac12}_{2,1}, \qquad |1+W_{\alpha}| \geq \delta > 0.
\end{split}
\end{equation}
Then our main nonexistence result in holomorphic coordinates is

\begin{theorem}\label{t:g-hol}
The above equation has no nontrivial holomorphic solutions $W$ which have regularity \eqref{W-reg}.
\end{theorem}
\begin{proof}
  We first directly dispense with the case $c=0$. For $c \neq 0$, 
on the other hand, we can improve the low frequency regularity of $W$.
 We begin with \eqref{W-reg} and the Moser estimates in $\dot
B^{\frac12}_{2,1}$ to conclude that $\Im W \in \dot
B^{\frac12}_{2,1}$, and thus $W \in \dot B^{\frac12}_{2,1} \subset
L^\infty$ (after possibly subtracting an appropriate constant).
Interpolating this with \eqref{W-reg} we obtain $W_{\alpha} \in L^2$. Now we
reiterate the same steps to succesively conclude that $\Im W \in L^2$,
then $W \in L^2$.  Since $W_\alpha \in L^2$, we also conclude that $W$
has limit zero at infinity.

Multiplying the equation \eqref{1g-re} by $(1+W_{\alpha})$ gives
\[
2g (1+ W_{\alpha}) \Im W =   c^2 \left( W_{\alpha} + \frac{\bar W_{\alpha}}{1+\bar W_{\alpha}}\right),
\]
which after projection yields the Babenko's equation
\[
-i g W + 2 g P[ \Im W  W_\alpha ] =   c^2 W_\alpha . 
\]
We rewrite this equation  in the form
\begin{equation}\label{bab-g}
g W = P[ i V W_\alpha], \qquad  V = c^2 +i g ( W -\bar{ W}).
\end{equation}
Then it remains to show that this problem admits no $H^1$ solutions.

The key remark is that the coefficient $V$ of $W_\alpha$ is purely real, and also bounded.
If there were no projection, we could choose a bounded increasing
function $\chi$, multiply the equation by $2 \chi \bar W_\alpha$, take the real part and integrate by parts to get
\[
0 = 2\Re \int \chi \bar W_\alpha W \, d\alpha  = - \int \chi' |W|^2\, d\alpha, 
\]
which yields $W=0$. In our case, we need to be more careful. We set $\chi$ to be a smooth
bump function which increases from $0$ to $1$, and for $r > 0$ we define 
\[
\chi_r := \chi(\alpha/r).
\]
Then, multiplying as above by $\chi \bar W_\alpha$ and integrating by parts on the left,  we have the following equality
\[
g\int \chi_r' |W|^2 \, d\alpha = -2 \Re  \int \chi_r \bar W_\alpha P( i V W_\alpha)\, d \alpha
= -2\Re \int \overline{[P,\chi_r] W_\alpha}  \cdot (i V W_\alpha) \, d\alpha ,
\]
where we have used the fact that $P$ is $L^2$ self-adjoint.  We rewrite this in the form
\[
g\int \chi'(\alpha/r)' |W|^2 \, d\alpha = -2 r \Re \int \overline{[P,\chi_r] W_\alpha}  
\cdot (i V W_\alpha) \, d\alpha.
\]

Here we have two favourable features. First, we have a commutator which yields the $r^{-1}$ gain
via the classical Coifman-Meyer estimate  \cite{cm},
\begin{equation}\label{CM}
\| [P,\chi_r] W_\alpha\|_{L^2} \lesssim \| \chi'_r \|_{L^\infty} \| W\|_{L^2} \lesssim r^{-1} \|W\|_{L^2}.
\end{equation}
This is almost enough but not quite. Secondly only the frequencies in
$W$ which are less than $r^{-1}$ affect the commutator.  Because of
this we claim that that the bound \eqref{CM} admits a qualitative improvement in the limit 
as $r \to \infty$:
\begin{lemma}\label{l:com}
Let $W \in L^2$. Then
\[
\lim_{r \to \infty} r \| [P,\chi_r] W_\alpha\|_{L^2} = 0.
\]
\end{lemma}
To conclude the proof of the theorem,  we let $r \to \infty$ in the last integral relation. By Lebesque
dominated convergence theorem, we  obtain
\[
\chi'(0) \int |W|^2 \ d\alpha \leq \lim_{r \to \infty} \int r \chi_r' |W|^2 \ d\alpha = 0,
\]
which implies $W = 0$ and indeed concludes the proof of the theorem.

\end{proof}

 It remains to prove the Lemma.

\begin{proof}[Proof of Lemma~\ref{l:com}]
  As mentioned above, the idea is to use the fact that primarily only
  the frequencies $\lesssim r^{-1}$ of $W$ contribute to the
  commutator.  Precisely, given a frequency threshold $\lambda > 0$ we
  separate  $W$ into a low and a high frequency part,
\[
W = W_{<\lambda} + W_{\geq \lambda}.
\]
For the contribution of the low frequency part we use directly \eqref{CM},
\[
\| [P,\chi_r] W_{\alpha,<\lambda} \|_{L^2}  \lesssim r^{-1} \|W_{<\lambda}\|_{L^2}.
\]
For the high frequency part we write
\[
 [P,\chi_r] W_{\geq \lambda, \alpha} =  \left[ P,(\chi_r)_{ \geq\frac{\lambda}{4} } \right] 
W_{\geq \lambda, \alpha },
\]
where we took the $\lambda /4$ truncation in $\chi_r$ to make sure
that the excluded frequencies in $\chi_r$ are strictly smaller than
the frequencies in $W$ and thus do not contribute to the commutator.
Here we have rapid decay as $r$ approaches infinity,
\[
\left|\partial_{\alpha} (\chi_r)_{\geq\frac{ \lambda}{4} }\right| \lesssim r^{-1} (\lambda r)^{-N},
\]
where $N$ is any fixed large positive integer.

Therefore using directly the Coifman-Meyer commutator estimate we obtain
\[
r \left\| \left[P,(\chi_r)_{\geq\frac{ \lambda}{4}} \right] W_{\geq \lambda, \alpha}\right\|_{L^2} \lesssim  r^{-1} (\lambda r)^{-N} \| W\|_{L^2} \to 0.
\]

Adding the low and high frequency contributions, it follows that 
\[
\limsup_{r \to \infty}  r \| [P,\chi_r] W_\alpha\|_{L^2} \lesssim \| W_{\leq \lambda}\|_{L^2}.
\]
Now we let $\lambda \to 0$ to conclude the proof. 
\end{proof}

\subsection{ Further comments on crested waves}

Here we show how our result in Theorem~\ref{t:g-hol} also precludes
the existence of crested solitary waves:

\begin{corollary}
  There are no crested wave solutions $W$ with 
\[
  \dfrac{W_{\alpha}}{1+W_{\alpha}} \in \dot B^{\frac12}_{2,1} \mbox{ and }   |1+W_{\alpha}| \geq \delta > 0, 
\]
  for the equation \eqref{1g-re} .
  
\end{corollary}
Here, for simplicity, by crested waves we mean waves which are smooth
except for finitely many angular crests. Certainly, there is room to refine this further
but we choose no to pursue it here.

 We next show how to prove
this corollary.  We will not use directly the result of Theorem~\ref{t:g-hol},
as near a crest $W_{\alpha}$ cannot be expected to have the regularity \eqref{W-reg}.
Instead, we will use the second part of the proof of the theorem, which shows 
that the Babenko equation~\eqref{bab-g} has no solutions $W$ with $W_\alpha \in L^2$.

To set the stage, observe  that the equation \eqref{1g-re} implies that 
\[
\Im W \leq h_0 :=\frac{c^2}{2g}.
\]
Here $h_0$ is the maximum possible height of a solitary wave with
speed $c$.  In the the region where $\Im W < h_0$, the coefficient $V$
of $W_\alpha$ in \eqref{bab-g} is strictly positive, which makes the equation
\eqref{bab-g} elliptic.  This in turn implies that $W_\alpha \in
C^\infty$ there. Thus the only way we can have a crest is at a point
$\alpha_0$ of maximum height $h_0$.

Away from the crests, the first argument in the proof of Theorem~\ref{t:g-hol}
combined with the hypothesis of the corollary still yields the regularity $W_\alpha \in L^2$.
So we need to investigate what happens near an angular crest.

Suppose that at some  point $\alpha_0$  we have an angular crest of angle $\theta \in (0,\pi)$.
Then near $\alpha_0$ the conformal map $Z$ must have the form 
\[
Z(z) \approx i h_0 + C (z-\alpha_0)^{\frac{\theta}{\pi}},
\]
and
\[
Z_\alpha = 1+W_\alpha \approx  C (\alpha-\alpha_0)^{\frac{\theta}{\pi}-1}.
\]

Inserting such an ansatz in the equations \eqref{1g-re} the leading singular
part  on the left hand side is $(z-\alpha_0)^{\frac{\theta}{\pi}}$, whereas on the right hand side, the
similar power is $(z-\alpha_0)^{2(1-\frac{\theta}{\pi})}$.  Matching
the two, we see that the only admissible angular crest corresponds to
the angle $\theta = 2\pi/3$ (i.e. Stokes type waves). 
But  if all the crests have $2\pi/3$ angles then
the condition $W_\alpha \in L^2$ is satisfied\footnote{More generally, we have $W_\alpha \in L^2$ 
iff all crest angles are in the range $\theta \in (\frac{\pi}2,\pi)$.},  so the existence of $W$
excluded by the second part of the proof of Theorem~\ref{t:g-hol}.

\section{Capillary waves}
\label{s:c}

Here we start with the algebraically equivalent  equations
\begin{equation}\label{1cc-re}
i \sigma \partial_\alpha \frac{1+\W}{|1+\W|} = c^2 \left[ \W + \frac{\bar \W}{1+\bar \W}\right],
\end{equation}
or 
\begin{equation}\label{1cc+}
- \partial_\alpha \frac{\W - \bar \W}{|1+\W|} = ic^2 \left[  \frac{ \W(2+ \W)}{1+ \W} + \frac{\bar \W(2+\bar \W)}{1+\bar \W} \right] .
\end{equation}

\begin{theorem}\label{t:c-hol}
The above equation \eqref{1cc-re}  has no nontrivial solutions $W$ with regularity \eqref{W-reg}.
\end{theorem}

\begin{proof}
  If $c = 0$ then from the second equation we immediately obtain $\Im
  \W = 0$ which yields $\W = 0$.  Suppose now that $c \neq 0$. We
  first improve the regularity of $\W$ both at low and at high
  frequencies.

At low frequencies we project in the first equation to obtain
\[
i \sigma  P \partial_\alpha \frac{1+\W}{|1+\W|} = c^2  \W .
\]
 Next we integrate once,
\[
i \sigma  P  \left( \frac{1+\W}{|1+\W|}-1 \right ) = c^2  W .
\]
Now we argue as in the case of gravity waves.
Since $\W \in \dot B^{\frac12}_{2,1}$, by Moser estimates  
this yields n $W \in \dot{B}^{\frac12}_{2,1}$.
Now by interpolation we get $\W \in L^2$, which after reiteration
 yields $W \in L^2$ and also $W \in L^\infty$. 

At high frequencies we work with the function
\[
T := \log (1+ \W) \in L^\infty \cap L^2,
\]
which is also holomorphic. 

From the first equation \eqref{1cc-re} we get $\Im T_\alpha \in L^2$, 
which shows that $T_\alpha \in L^2$, and eventually $\W_\alpha \in L^2$.
Differentiating and repeating the argument it follows that $\W \in H^\infty$.

Combining the low and high frequency information, it now suffices to prove 
that the equation \eqref{1cc-re} has no solutions $\W \in H^\infty$ with 
$|1+\W| \geq \delta > 0$. For this we switch to polar coordinates, denoting 
\[
T = \log (1 +\W) =: U+ iV  \in H^\infty.
\]
Substituting this into \eqref{1cc-re}  we have
\[
-\sigma V_\alpha e^{iV} =  c^2 ( e^{U+iV} - e^{-U+iV} ) ,
\]
or equivalently 
\[
-\sigma V_\alpha = 2 c^2 \sinh U  .
\]
We rewrite this in the form
\[
H U_\alpha = 2 c^2 \sinh U .
\]
Then we repeat the argument for gravity waves, multiplying by $\chi_r U_\alpha$ and integrating by parts,
using the same commutator bound as in Lemma~\ref{l:com} (which holds equally for the Hilbert transform).
This yields $U = 0$, and then $V=const$. Then we must have $\W = const$ and further $\W = 0$ since $\W \in L^2$.

\end{proof}

\bibliographystyle{abbrv}
\bibliography{refs}

\end{document}